\documentclass[a4paper,11pt]{amsart}

\usepackage{amssymb,amsbsy,amsmath,amsfonts,amssymb,amscd}
\usepackage{color}
\input xy
\xyoption{all}

\usepackage[utf8]{inputenc}

\usepackage{tikz-cd}

\theoremstyle{plain}
\newtheorem{thm}{Theorem}[section]
\newtheorem{theorem}[thm]{Theorem}

\newtheorem{lemma}[thm]{Lemma}
\newtheorem{corollary}[thm]{Corollary}
\newtheorem{proposition}[thm]{Proposition}
\theoremstyle{definition}
\newtheorem{remark}[thm]{Remark}

\newtheorem{defin}[thm]{Definition}

\newtheorem{example}[thm]{Example}

\numberwithin{equation}{section}
\newcommand{\sA}{{\mathcal A}}
\newcommand{\sB}{{\mathcal B}}

\newcommand{\sD}{{\mathcal D}}

\newcommand{\sF}{{\mathcal F}}

\newcommand{\sH}{{\mathcal H}}

\newcommand{\sK}{{\mathcal K}}

\newcommand{\sR}{{\mathcal R}}
\newcommand{\sS}{{\mathcal S}}

\newcommand{\sX}{{\mathcal X}}
\newcommand{\sY}{{\mathcal Y}}
\newcommand{\sZ}{{\mathcal Z}}
\newcommand{\PP}{\ensuremath{\mathbb{P}}}

\newcommand{\CC}{\ensuremath{\mathbb{C}}}

\newcommand{\ZZ}{\ensuremath{\mathbb{Z}}}

\newcommand{\hol}{\ensuremath{\mathcal{O}}}

\newcommand\la{\lambda}
\newcommand\Lam{\Lambda}
\newcommand\s{\sigma}

\newcommand\al{\alpha}
\newcommand\be{\beta}
\newcommand\Ga{\Gamma}
\newcommand\De{\Delta}

\newcommand\de{\delta}

\DeclareMathOperator{\Pic}{Pic}

\newcommand{\ra}{\ensuremath{\rightarrow}}

\def\eea{\end{eqnarray*}}
\def\bea{\begin{eqnarray*}}

\newcommand\dual{\mathrel{\raise3pt\hbox{$\underline{\mathrm{\thinspace d
\thinspace}}$}}}
\newcommand\qe{\ifhmode\unskip\nobreak\fi\quad $\Box$}       % box for QED

\def\BOX{\hfill\lower.5\baselineskip\hbox{$\Box$}}
\newtheorem{theo}{Theorem}[section]

\newtheorem{remarkk}[theo]{Remark}

\newenvironment{dedication}
        {\begin{quotation}\begin{center}\begin{em}}
        {\par\end{em}\end{center}\end{quotation}}

\title [General birationality and Hyperelliptic Jacobians ]{General  birationality  and Hyperelliptic Theta divisors}
\author{Fabrizio Catanese}
\address {Mathematisches Institut der Universit\"at Bayreuth\\
NW II,  Universit\"atsstr. 30\\
95447 Bayreuth}
\email{fabrizio.catanese@uni-bayreuth.de}
\address{  Korea Institute for Advanced Study, Hoegiro 87, Seoul, 
133--722.}


\date{\today}

\begin{document}

\begin{abstract}

 We first state a condition ensuring that  having a birational map onto the image
is an open property for families of irreducible normal non uniruled varieties.  We give then some criteria to
ensure general birationality for a family of rational maps, via specializations.

Among the  applications is a new   proof of the main result  of \cite{Cat-Ces} that,   for a general  pair $(A,X)$ of an (ample) Hypersurface $X$ in an Abelian Variety $A$,
the canonical map  $\Phi_X$ of $X$   is birational onto its image if the polarization given by $X$ is not principal. 

The proof is also based on a careful study   of the Theta divisors of the Jacobians of Hyperelliptic curves, and some related  geometrical constructions.  We investigate  these here also in view  of their beauty and of their  independent interest, as they lead to a description of the rings  of Hyperelliptic theta functions.

\end{abstract}
%\vskip 0.5\baselineskip

\maketitle

\tableofcontents
\begin{dedication}

\end{dedication}

%%%%%%%% pag  2 inizio %%%%%%%%%%%%%%%%%%

%%%%%%%%%%%%%%%%%%%%%%%%%%%
%%%%%%%%%%%%%%%%%%%%%%%%%%%%%%%%%%%%%%%%%%%%%%%%

\section{Introduction}

One of the main general problems in algebraic geometry is the study of the canonical and pluricanonical maps
of varieties of general type, especially the problem of establishing their birationality, see for instance
\cite{cm}, \cite{cil-santacruz}, \cite{chen1}, \cite{chen3}, just to name a few items.

We describe here a simple and  relatively  general method for establishing birationality of a rational map  for the general variety in a family, via specializations, see Theorems \ref{birational} and  \ref{blocks}.

The applications can be many  (see \cite{sing-bidouble} in the case of surfaces), but we focus here on the use of this method  for the problem which was our original motivation, and   we exhibit  a new self-contained proof  (Theorem \ref{genbirat}) of the main result of \cite{Cat-Ces}. The present proof  couples this method with an interesting study,   of  the geometry of Hyperelliptic Jacobians and of  some of their unramified cyclic coverings.

In the course of doing this we establish some general results on the graded rings of Theta Functions on Hyperelliptic Jacobians, see Theorems \ref{thetapullback} and  \ref{theta}.

\section{Openness of birationality and general birationality}

 As already mentioned, our present main problem is: given a family of varieties $\{ X_t\}_{t \in T}$, and a family 
of morphisms $ f_t : X_t \ra Y_t$ (respectively, rational maps), when can we conclude, from the fact that  $f_0$
is birational onto its image $Y_0$, that, for general $t$, $f_t$ is birational onto $Y_t$?

Let us start with a negative example: let $X$ be a hypersurface in $\PP^N$
of degree $d$, let $P$ be   a point,  $P \in  \PP^N$, and consider the projection with centre $P$,
$ \pi_P : \PP^N \setminus{P} \ra \PP^{N-1}$.

If the hypersurface $X_0$ has   multiplicity $d-1$ at the point $P$, then
$\pi_P$ induces a birational map between $X_0$ and $\PP^{N-1}$,
but for a general $X$ the projection is not birational, having degree equal to 
$  d- mult_P(X) $, which is $\geq 2$ as soon as $mult_P(X) < d-1$.

The important feature of this example, which motivates the assumption
in the following  theorems, is  that $X_0$ is a uniruled variety, indeed it is a rational variety: and this must be avoided.

The next example, instead, clarifies the hypotheses needed for the validity of
an  assertion made  in the first version of this paper (see for instance the next Proposition \ref{smooth}).

\begin{example}\label{counterex}
Consider in $\PP^N \times \PP^1$ the following family 
$$ \sX : = \{ (x, (\la_0, \la_1)) | \la_0^m f(x) + \la_1^m g(x) = 0 \} ,$$
where the Hypersurfaces $X_0 : = \{ f(x)=0\}$ and $X_\infty : = \{ g(x)=0\}$
intersect transversally, $X_\infty $ is smooth, while $X_0$ has only one isolated
singular point $P$ of multiplicity $m$, and is of general type if $d : = \deg (f) = \deg (g) \geq  N+2 + m$.

An elementary calculation shows that $Sing (\sX) = \{ (P, (1,0)\}$, a point of multiplicity equal to $m$.

Hence $X_0, \sX$ are normal (being hypersurfaces in a smooth manifold).

Blowing up the only singular point, we get 
$$ \sZ \ra \sX \subset \PP^N \times \PP^1 , p :  \sZ   \ra \PP^1,$$
and the fibre $Z_0$ consists of the union of the blow up $X_0'$ of $X_0$ in $P$, together with the hypersurface 
$Z'_0$ in the exceptional $\PP^N$, 
$$ Z_0' : = \{ \phi_m(x_1, \dots, x_n) + \la^m g (P) =0\},$$ 
where we assume that $P = (1, 0, \dots, 0)$ and that $\phi$ is the leading term of the Taylor development of $f$ at $P$.

For  $m \geq N+ 2$ and $\phi$ general, $ Z_0'$ is a smooth variety with ample canonical system,
and $X_0' \cap Z_0' = \{ \phi_m(x_1, \dots, x_n)=0\}$, the exceptional divisor of $\pi : X_0' \ra X_0$.
 
\end{example}

The following Proposition is a direct  consequence of Hironaka's II Main Theorem in \cite{hironaka}
\begin{proposition}
\label{smooth}
Assume that we have a 1-dimensional  projective  family $p : \sS \ra T$  where $\sS$ is smooth of dimension $n+1$, $ T$ is a smooth connected curve, $0 \in T$, and  we have a rational map
$$ f : \sS \dasharrow  \PP^N.$$
Then there exists a modification $ \pi : \sZ \ra \sS$ such that, setting 
$$F : = f \times p:  \sS \dasharrow   \PP^N \times T,$$ and denoting  by $\sY$  the closure of the image of $F$,

(i) $\sZ$ is smooth, 

(ii) $F' : = F \circ \pi$ becomes a morphism  $F' : \sZ \ra \sY$,

(iii) all the fibres of $p' : = p \circ \pi : \sZ \ra T$ consist of the union of the strict transform $S'_t$ of $S_t : = p^{-1} (t)$
with other ruled components.

(iv) In particular, if the indeterminacy locus of $f$ is contained in $S_0$, then $S'_t = S_t$ for $t \neq 0$.

(v) It follows that, if $\Ga \ra \sS \times  \PP^N \times T$ is the normalization of the graph of $F$, then the fibres 
$\Ga_t$ consist of the strict transform of $S_t$ plus some uniruled components.

(vi) Shrinking $T$, we may assume that in (iii) and (v) other uniruled components only occur for $t=0$.

\end{proposition}

\begin{proof}
A preliminary observation is that, since $\dim (\sS) = n+1$,  the indeterminacy locus of $F$ does not contain any fibre $S_t$.

Assertions (i) and (ii)  follow from Hironaka's II Main Theorem of \cite{hironaka} (see especially page 140, and the affirmative answer to Question (F), (iii), the assertion that $f_r$ is a morphism, and that the centres $D_i$ are smooth
and contained in the indeterminacy locus of $f_i$)  ensuring that, via a sequence of blow ups with smooth centres, we get $ \pi : \sZ \ra \sS$  such that  the rational map  $F' : = F \circ \pi$ becomes a morphism on $\sZ$.
 
For (iii) we  just need to observe that, 
 if we blow up a submanifold $W$ of a manifold $M$, then the blow up $\tilde{M}$ contains as exceptional divisor 
the ruled manifold $\PP(N_{W|M})$. 
Hence the exceptional divisors are all ruled, hence so are the new irreducible components of the fibres of $p'$
(as they are divisors in $\sZ$ by our assumptions).

(iv) follows since the centres of the blow up are contained in the inverse image of the indeterminacy locus in $\sS$.

(v): since $\sZ_t $ surjects onto $\Ga_t$, the other components of $\Ga_t$ are  images of a ruled manifold, hence they are uniruled.

(vi) first of all, the set of $t\in T$ such that $S_t$ is not irreducible is closed;
furthermore, since there is only a finite number of exceptional divisors, there is only a finite number of $t$ such that the fibres $\sZ_t$ and $\Ga_t$ are not irreducible. So we omit these two finite subsets of $T$.

\end{proof}

\begin{remark}
In view of Hironaka's extension \cite{hiro-2} of the resolution results to complex spaces, one can replace the hypothesis that we have a projective family by the hypothesis that we have a proper family.

\end{remark}

\begin{theorem}\label{birational}{\bf (Openness of birationality)}

 Let $p : \sX \ra T$ be a 1-dimensional (flat) family of reduced projective subschemes of dimension $n$ (i.e., $\sX$ is irreducible
and T is a smooth connected
curve, $0 \in T$) such 
that $X_0 = p^{-1} (0)$ contains a unique  irreducible component $X_0''$ which is  not uniruled.

Let  $ f : \sX \dasharrow  \PP^N$ be a rational map 
such   that $f_0 : X_0'' \dasharrow Y'_0$ is birational to its image.

Assume moreover
 
(**) 
setting  $F : = f \times p:  \sX \dasharrow   \PP^N \times T$, letting  $\sY$ be the closure of the image of $F$, and 
letting $\Ga$ be the normalization of the graph of $F$, then 
 the fibres 
$\Ga_t$ are irreducible for $t \neq0$, while $\Ga_0$ consists of the strict transform of $X_0$ plus some uniruled components.

Then $f_t : X_t \ra Y_t$ is birational to its image for all $t$ in a neighbourhood of $ 0 \in T$.
\end{theorem}
\begin{proof}
Clearly $\sY$ is irreducible and it has  dimension $n+1$ since its fibre 
$Y_0$ over $0$ contains $Y'_0$ which has
dimension $n$; the fibre $Y_t$
over $t \neq 0$  contains  the image $Y'_t$ of $X_t$ which by assumption is    irreducible.

The rational map $F$ induces  a surjective morphism $F' : \Ga \ra \sY$.

$\Ga$ is irreducible of dimension $n+1$, and the central image $Y_0$ is the image
of $\Ga_0$ under a proper map, and contains $Y'_0$ as a component, since the strict transform $X'_0$ of $X_0''$ is a component of $\Ga_0$.

The other components of $\Ga_0$ are uniruled, hence they cannot dominate the component $Y'_0$,
 which is not uniruled.

 Hence the general point $y \in Y_0'$ is in the image of only one point $x$,
 this point $x $ lies in $X'_0$, and the map $F'$  is of maximal rank in $x$,
 hence a formal isomorphism with its image: because $f_0$ is a local isomorphism and $p$
 is a submersion at $x$ (in particular there is no ramification of $F'$ at $x$).
 
 Consider now a local holomorphic section $\Sigma$ of $\sY \ra T$ passing through
 $y$ (which is a smooth point of $\sY$ and of the fibre $Y_0$, since $F'$ and $p$  are local submersions at $x$).
 
 If the map $f_t$ were non birational for all $t$, then $f'_t : X'_t \ra Y'_t$
would have  positive degree, and would be  \'etale outside of a branch locus $\sB_t \subset Y'_t$.
 
 We have seen  that if $y \in Y_0'$ is chosen general,  it is not contained in the closure
  $\sB$ of the branch loci: since there is no ramification at $x$.
  
  Therefore the inverse image of $\Sigma$ consists of holomorphic arcs,
  in a number strictly greater  than one,  
  of which only one contains $x$ in its closure, while the other arcs
  tend to a point $z$  in $\Ga_0$ different from $x$. 
  
 The conclusion is that $F'(z)=y, z \neq x$, and we have
  reached the desired contradiction: hence we have proven that $f_t$ is birational.
   
 \end{proof}
 
 \begin{remark}
 
 The above Theorem and the following ones can be stated in more general situations.

(i)   We can consider more generally \footnote{Thanks to Thomas Peternell for asking this question.}
 a family $\sX \ra T$ of compact complex spaces, and  a meromorphic
map $ f : \sX \dasharrow  M$, where $M$ is a complex manifold: the    above proof   works without any change.

(ii) The same theorem is true for a projective family  over an algebraically closed  field of any  characteristic, if we assume that $f_0$ is separable and birational on $X_0''$.

We have in fact  that $ F' : \Ga \ra \sY$ proper, hence there is a closed set $B \subset \sY$ with nontrivial complement
$ \sY \setminus B$, such that, over $ \sY \setminus B$, $F'$  is finite with all the fibres of cardinality $d$.

If $f_t$ is not birational, then $ d \geq 2$. Since we have shown that $y \notin B$, and that $F'^{-1} (y)$
is a single point with multiplicity $1$, it follows then that $d=1$, a contradiction.

 (ii) The theorem was applied as self evident in the case of canonical maps of algebraic surfaces in \cite{sing-bidouble},  but its use  was criticized as non self evident in \cite{liedtke}.  All details of the proof are now following from Proposition \ref{smooth}, (v), applied to the family of smooth minimal models of surfaces of general type,  and 
 from Theorem \ref{birational}. 
  
 \end{remark} 
 
   Before moving to a more general Theorem, we need to state a simple group theoretical result.

\begin{lemma}\label{comp-mon}
Given finite groups $ \Pi_X < \Pi' <  M_0$, where the maximal normal subgroup of $M_0$
contained in  $\Pi_X$ is the identity, let us  set: 

(1) $M_0^{\nu} : = M_0 /K$, where
$K$ is the maximal normal subgroup contained in $\Pi'$, so that 

(2) 
 $M_0$ acts faithfully on the coset space $ \sF_X : = M_0 / \Pi_X$, whose cardinality will be denoted by $d$, 
 
 (3)  $M_0^{\nu}$  acts
on the coset space $ \sF^{\nu} := M_0 / \Pi'$, whose cardinality will be denoted by $m$,

(4) 
$M_t : = \Pi' / K_t$ acts on $\sF_t: = \Pi' / \Pi_X$, where
$K_t$ is the largest normal subgroup of $\Pi'$ contained in $\Pi_X$.

Then $ d = \de m$, where $\de$ is the cardinality of the set $ \sF_t  = \Pi' / \Pi_X$.

And  the action of $M_0$ preserves the blocks corresponding
to the $m$ elements of $\sF^{\nu}$, 
$$ \sF_X = \cup_{[c] \in \sF^{\nu}} \ c  \Pi' / \Pi_X.$$

Hence we have exact sequences
$$ 1 \ra K \ra M_0 \ra M_0^{\nu} \ra 1,$$
$$ 1 \ra  K_t  \ra \Pi' \ra M_t \ra 1,$$
and, setting $G : = \Pi' / K$, 

$$ G < M_0^{\nu} .$$

\end{lemma}.

With a similar proof to Theorem \ref{birational}, we obtain the following more  general result which is  useful
 for applications.
 
 \begin{theorem}\label{blocks}
 
  Let $p : \sX \ra T$ be a 1-dimensional (flat) family of  projective varieties  of dimension $n$, with $\sX$  irreducible,  T  a smooth connected
curve, $0 \in T$,  such 
that $X_0 = p^{-1} (0)$ is irreducible normal.

Let  $ f : \sX \dasharrow  \PP^N$ be a rational map 
such   that $f_0 : X_0 \dasharrow Y'_0$ is of degree $d$  to its image $Y'_0$,
which is not uniruled.

Assume moreover
 
(**) 
setting  $F : = f \times p:  \sX \dasharrow   \PP^N \times T$, letting  $\sY$ be the closure of the image of $F$, and 
letting $\Ga$ be the normalization of the graph of $F$, then 
 the fibres 
$\Ga_t$ are irreducible for $t \neq0$, while $\Ga_0$ consists of the strict transform of $X_0$ plus some uniruled components.

 Then

(i) $f_t : X_t \ra Y_t$ has degree $\de$   onto its image for all $t$ in a neighbourhood of $ 0 \in T$, 
with $\de$ dividing $d$, so that  we may write  $ d = m \de$.

More precisely,  $f_0 : X_0 \dasharrow Y'_0$ admits a factorization as $\nu_0  \circ  F''_0$, 
where $\nu_0$ has degree $m$, and the monodromy group  $M_0 \subset \frak S_d$ of $f_0$
is thus related to  the monodromy group of $f_t$, 
$M_t \subset \frak S_{\de}$,
and the  monodromy  group $M^{\nu}_0 \subset \frak S_{m}$ of $\nu_0$, as in 
the statement of Lemma \ref{comp-mon}.

(ii) In particular, if the monodromy group $M_0 \subset \frak S_d$ is primitive (that is,
there is no nontrivial partition of $\{1, \dots, d\}$ which is $M_0$-invariant)
then either the general $f_t$ is birational ($m=d$) or it has degree $ \de = d$ ($m=1$).
\end{theorem}

\begin{proof}
Using the same notation as in the proof of Theorem \ref{birational}, 
we are then in a similar  situation.

 The general point $y \in Y_0$ is in the image of exactly $d$ smooth points $x_1, \dots, x_d$ of $\Ga$
 which  lie in $X'_0$, and the map $F'$  is of maximal rank in each $x_i$.
 
 What may now change is that $y$ could be contained in the singular locus
 of $\sY$, and there may be $m$  smooth branches of $\sY$ passing through $y$.
 
 Therefore, we take the normalization $\nu : \sY^n \ra \sY$, and notice that
 we have a factorization of $F'$ as $\nu \circ F''$, where $F'' : \Ga \ra \sY^n$.
 We observe then that the morphism $F'' | X'_0$ will have degree $\de$ onto its image,
 where $ d = \de m$.
 
 Hence the $d$ points are grouped in $m$ subsets, corresponding to the inverse images of
 the points $y_1, \dots, y_m$ lying over $y$ in  $\sY^n$, and the previous argument 
 using the  local holomorphic sections $\Ga_i$  of $\sY^n  \ra T$ passing through
 $y_i$ for $i=1, \dots , m$ shows
 that the degree of $f_t : X_t \ra Y_t$ equals $\de$.
 
 Assertion ii) follows right away because, if $ m \neq 1,d$,  then there is a partition of  
  $\{1, \dots, d\}$ in $m$ subsets which
are permuted by $M_0$.
 
Now, the monodromy of $f_t : X_t \dasharrow Y_t$ will be the same as the one of 
$F'' : X'_0 \ra Y^n_0$, and since  $f_0 : X'_0 \dasharrow Y_0$ is a composition, it follows that
the monodromy of $f_0$ is  as claimed, in view of the previous  Lemma \ref{comp-mon},
where we divide the respective fundamental groups by the largest normal subgroup
of the fundamental group of the open set of $X_0'$ where all coverings are
unramified, so that $M_0$ and $M_t$ are the monodromy groups we are talking about.

%\begin{remark}
%It is not clear that it must be $ K_t < K$, because 
%$$ K = \cap_{\ga \in M_0} (\ga \Pi' \ga^{-1}) , \ \  K_t = \cap_{\de \in \Pi'} (\de \Pi_X \de^{-1}).$$
%\end{remark}
  
\end{proof}

\begin{corollary}
Let $p : \sX \ra T$ be a  family of projective varieties of dimension $n$, where 
T is  smooth and connected. Assume moreover  that we are given  a rational map $ f : \sX \dasharrow  \PP^N$ which is a morphism for $t \in V$, where $V$ is  an open set
$V \subset T$.

(I)  Assume that for a general point $t \in T$ there are several 1-parameter specializations,
for $j =1, \dots, r$,
with base $T_j$ containing  $t$   and $t_j  \in T$, 
of the fibre $X_{t} = p^{-1} (t)$ to  the fibre $X_{t_j} $. Assume that these are,  as in Theorem \ref{blocks}, such 
that  $X_{t_j} $ is irreducible and normal with monodromy in $\frak S_{d(t_j)}$, and that moreover,
writing  $d_j : = d(t_j)  $,  we have  
$$GCD \{ d_j | j =1, \dots, r\} =1.$$

 Then, for general $t$,   $f_{t}$ is birational.  
 \smallskip
 
 (II) The same conclusion holds if there are two 1-parameter specializations, one 
 such that the monodromy group $M_0 \subset \frak S_{d_1}$ is primitive,
 the other such that $d_1$ does not divide $d_2$.
\end{corollary}

\begin{proof}

We denote as above by $\de$ the degree of $f_{t}$ for general $t$.

Our claim(I),   in the notation 
of Theorem \ref{blocks}, is that $\de=1$, which is obvious
since, by (i) of theorem \ref{blocks},  $ d_j  = \de m(t_j)$ therefore  $\de$ divides all the integers $d_j$, 
hence their GCD.

To show (II), simply apply (i) and (ii) of Theorem \ref{blocks}: in fact the general degree $\de$ is
either $1$ or $d_1$ by virtue of (ii), while (i) implies that $\de | d_2$. 

\end{proof}

\begin{remark}
One can obtain other more complicated criteria using the above exact sequences of groups.

But, if $M_0$ is Abelian, then $\Pi_X=0$, $K = \Pi'$, $K_t=0$,  hence $M_t = \Pi'$.

 If all specializations found yield a group $M_{t_0(j)}$ which is Abelian, then
 a criterion of triviality of $M_t$ follows  from a  criterion similar
to the above Corollary, analyzing the primary decompositions of all the
groups $M_{t_0(j)}$.

If we get one specialization such that one $M_{t_0}$ is Abelian,
then $M_t$ is Abelian, and is, for any other specialization,  a quotient of the Abelianization of $\Pi'$
by the image of $\Pi_X$.

\end{remark}

\begin{remark}
The main conjecture raised in \cite{Cat-Ces} is that the canonical map of a general pair
$(A,X)$ of an ample hypersurface in an Abelian variety is an embedding if the Pfaffian of the Polarization given by $X$
is at least $ dim(X) + 2$.

Also for this purpose it would be useful to establish in a similar way some criteria guaranteeing  
`general embedding', that is, embedding for  a
general variety in a family.
\end{remark}

\section{Theta divisors of Hyperelliptic curves}

We begin with a quite  elementary result in group theory.

\begin{lemma}\label{subgroups}
Consider the {\bf Group $G$ of the Hypercube}, namely the  
natural semidirect product (induced by coordinates permutation)
$$ G : = (\ZZ/2)^n \rtimes \frak S_n = : K \rtimes H.$$

Then

(i)  the only intermediate subgroups $H'$, with $ H< H' <G$,
and different from $H$, $G$, are just two subgroups $H_1, H_2$, of respective indices
$$ [H_1 : H] =2,  [G : H_2]=2.$$

(ii) the largest subgroup $H'' < H$ which is normal in $G$ is the identity subgroup.

\end{lemma}

\begin{proof}
For $v \in K$, $\s \in H$, we write $\s(v) : = \s v \s^{-1}$. 

For instance,
$\s(e_i) = e_{\s(i)}$.

If $H'$ is as in (i) and $H' \neq H$, then $H' \cap K =:V$ is then an $H$-invariant subspace.
And conversely, if $V$ is $H$-invariant, then $ V H$ is a subgroup,
because  $$v_1 \s v_2 \tau = v_1 \s(v_2) \s \tau.$$

Then assertion (i)  follows  from the

{\bf Claim:} {\em The only $\frak S_n$-invariant subspaces of $K$ are :}
$$ \{0\}, K , (\ZZ/2) e , e^{\perp}, {\rm where } \ \ e : = \sum _1^n e_i.$$

{\em Proof of the claim:} it is obvious that the four above subspaces are invariant.

For such an invariant subspace $V$, assuming that $V \neq 0$, consider a vector $v$ of minimal weight
$ w(v) : = | \{ i | v_i \neq 0\}|$. Denote by $w$ the minimal weight: if $w=1$, then
$\exists i $ such that $e_i \in V \Rightarrow V = K$.

Otherwise, we may assume, after a basis change,  that $ v = e_1 + \dots + e_w$.

If $ w = n$, we get that $ V = (\ZZ/2) e$. If instead $w < n$, then there is a $\s$
such that $\s(v) =  e_2 + \dots + e_{w+1}$, hence $ v  + \s(v) = e_1 + e_{w+1}$,
hence $w=2$. 

Then $e_1 + e_2 , e_2 + e_3 , \dots, e_{n-1} + e_n \in V$, hence $V$ is
an invariant hyperplane. Hence $V$ is orthogonal to a projectively invariant vector,
and we see that $V = e^{\perp}$.

Passing to (ii), 
$$H'' = \cap _{v \in K} v H v^{-1} = \cap _{v \in K} [ H \cap v H v^{-1} ].$$

Now,  
$$\s \in H \cap v H v^{-1} \Leftrightarrow \exists \tau \in H, \s = v \tau v^{-1} 
\Leftrightarrow \exists \tau \in H,\s \tau^{-1} = v \tau (v^{-1}).$$

 Since the last equality is between an element of $H$ and one of $K$,  this means that $\s = \tau$ and $ v = \s (v)$; if this is to hold for each $v \in K$, then $\s$
is the identity.

\end{proof}

We now come to an important  geometrical occurrence of the group $G$ of the Hypercube.

Let $C$ be a Hyperelliptic curve of genus $g$, and let $\psi: C \ra \PP^1$ be 
the canonical double cover (such that the canonical map $\phi$ of $C$
is the composition $\phi = v_{g-1} \circ \psi$, where $v_{g-1}  : \PP^1 \ra \PP^{g-1}$
is the Veronese embedding of $\PP^1$ as a rational normal curve of degree $g-1$).

Then, setting $Y : = C^n$, the group $G$ acts on $Y= C^n$, and we have 
 the following commutative diagrams:

	\[
\begin{tikzcd}
	C^n = Y \arrow{d}{p}\arrow{r}{\Psi} & Y /K = C^n/ (\ZZ/2)^n = (\PP^1)^n \arrow{d}{\pi}\\
	C^{(n)} : = C^n/ \frak S_n  = Y/ H \arrow{r}{\nu}& Y/G  = (\PP^1)^n /  \frak S_n = \PP^n
\end{tikzcd}
\]

It is well known that for $n=g$ we have Jacobi inversion, that is, $C^{(g)} $ has a
surjective birational morphism (the Abel Jacobi map) to the Jacobi variety $Jac(C) \cong Pic^g(C)$, while for
$n=g-1$, again via the Abel Jacobi map, $C^{(g-1)} $ has a  birational morphism onto
the Theta divisor $\Theta_C \subset Jac(C)$. We shall need to make these birational  statements
more precise.

We have the following classical result, due to Andreotti \cite{andreotti}.

\begin{theorem}\label{andreotti}
If $C$ is a Hyperelliptic curve, then $\nu  : C^{(g-1)}  \ra \PP^{g-1} $
is the composition of the birational Abel-Jacobi map
$\al_{g-1} : C^{(g-1)}  \ra \Theta_C$ with 
the Gauss map $\mu_C$ of $\Theta_C$, and $ p  :  C^{g-1}  \ra C^{(g-1)} $ yields  the Galois closure of the Gauss map.

For a non hyperelliptic curve $C$, letting $\phi$ be the canonical map $\phi : C \ra  \PP^{g-1} $, the composition $\mu_C \circ \al_{g-1}$ is the $g-1$ secant map of $\phi(C)$,
and the branch locus of the Gauss map is the dual variety $\phi (C)^{\vee}$ of the canonical curve in  $(\PP^{g-1})^{\vee} $.

The monodromy group of the Gauss map equals the monodromy group of the
canonical curve $\phi (C)$, the symmetric group $\frak S_{2g-2}$.
\end{theorem}
\begin{proof}
As shown by Andreotti, the map $\pi \circ \Psi$ is given as follows:
$$ (P_1, \dots, P_{g-1}) \mapsto \phi(P_1) \wedge \dots \wedge \phi (P_{g-1}) \in  (\PP^{g-1})^{\vee} ,$$
where $\phi$ is the canonical map $\phi : C \ra  \PP^{g-1} $, which is indeed the projective derivative $D(\phi)$ of the Albanese map= first Abel Jacobi map $\al : C \ra Jac(C)$.

On the other hand, the Gauss map associates to a point $x \in \Theta_C$,
$$ x = \al_{g-1} ( P_1 + \dots +  P_{g-1}) = \al ( P_1)  + \dots +  \al (P_{g-1})$$
 the Hyperplane spanned by $ \phi(P_1) ,  \dots ,  \phi (P_{g-1})$
 since $\phi$ is the projective derivative of $\al$.
 This shows that $\nu$ factors as claimed through the Gauss map $\mu_C$.
 
 The  assertion on the Galois closure follows now from Lemma \ref{subgroups}.
 
 See \cite{andreotti} and \cite{acgh} page 111 for the last assertions.
 
\end{proof}

The fact  that the degree of $\mu_C$ equals to $2^{g-1}$ follows  algebraically
since $ \pi \circ \Psi = \nu \circ p$, but also geometrically since each hyperplane intersects $\phi (C) $, image of $\PP^1$ through the Veronese map 
of degree $(g-1)$, in exactly $(g-1)$ points. 

For a non hyperelliptic curve the degree equals $ 2g-2 \choose{g-1}$.

For a general Theta divisor in a principally polarized
Abelian variety  the degree of the Gauss map is instead equal to $g! $.

For more general Jacobians, the Gauss map of the Theta divisor
is a rational map whose degree was studied in   \cite{cgs}.

\begin{proposition}\label{branch}
If $C$ is a Hyperelliptic curve the  map $\nu  : C^{(n)} \ra P: = \PP^{g-1}$
has a branch locus $\sB$ which set theoretically equals the 
union of $\De$, the Discriminant Hypersurface for divisors in $\PP^1$ 
of degree $n$ (the dual variety
of the rational normal curve $\Ga = \Ga_{n}$), with $2g+2$ Hyperplanes 
$H_1, \dots , H_{2g+2}$, where, if $z_i$ is a branch point of $ C \ra \PP^1$, then
$H_i $ corresponds to the divisors containing $z_i$.
Moreover $\De$ occurs with multiplicity $ 2^{n-1}$ in the branch locus,
 while the divisors $H_j$ occur with multiplicity $ 2^{n-2}$.

The  map $\nu$ factors exactly through  two intermediate coverings:

i) $ C^{(n)} \ra Z_n  : = C^{(n)} / \iota$, where $\iota$ is the hyperelliptic involution; 

ii) $ Z_n \ra \tilde{Z}_n$, where $ \tilde{Z}_n \ra P$ is the double cover 
branched on the union $\sH$ of the hyperplanes $H_1, \dots, H_{2g+2}$.

\end{proposition}

\begin{proof}
In view of Lemma \ref{subgroups} the main remaining point  to show is that the branch locus 
is as stated.

The ramification  locus of $\Psi : C^n \ra (\PP^1)^n$
equals the union of the divisors 
$$R_C(i) : = \{ (y_1, \dots, y_n) | y_i \in  R_C \},$$
where $R_C = \{ p_1, \dots, p_{2g+2}\}$ is the ramification divisor of $ \psi : C \ra \PP^1$.
These divisors are permuted by $\frak S_n$, and their image in 
$(\PP^1)^n$ equals 
$$B_C(i) : = \{ (x_1, \dots, x_n) | x_i \in  B_C\},$$
where $B_C = \{ z_1, \dots, z_{2g+2}\}$ is the branch divisor of $\psi$.

Whereas the ramification of $\pi : (\PP^1)^n \ra \PP^n$ consists of the 
fixpoints for some nontrivial element of $\frak S_n$, and its image is the discriminant 
hypersurface $\De$ consisting of the nonreduced divisors on $\PP^1$, that is, the divisors  $x_1 + \dots + x_n$ where
the points $x_i$ are not distinct. $\De$ is irreducible, being the image of 
$$ \PP^1 \times (\PP^1)^{n-2} \cong  \{ (x_1, x_1, x_3 , \dots, x_n)\}.$$

Hence the branch locus of $\nu \circ  p = \pi \circ \Psi$ is equal to the union of
$\De$ and of hyperplanes $H_1, \dots, H_{2g+2}$, where $H_i$ consists of the 
effective divisors in $\PP^1$ containing $z_i$.

$H_i$ intersects $\De$ in the linear space of codimension $2$
consisting of the divisors which are $\geq 2 z_i$, and in a smaller discriminant $\De'_i$ 
consisting of divisors which are 
the sum of $z_i$ with a nonreduced divisor. 
%Hence $H_i$ is tangent to $\De$.

On the other hand, the branch locus of $p : C^n \ra C^{(n)}$
equals the discriminant $\De_C$, consisting of nonreduced 
effective divisors   of degree $n$ on $C$.

$\De_C$ maps then to $\De$ with degree $2^{n-1}$, since for general
$x_1$ and general $x_3, \dots, x_n$  the inverse image of $ 2 x_1 +  x_3 + \dots + x_n$
consists of $2^{n-1}$ divisors.

While the inverse image of the ramification of $\pi$ contains the 
$\frak S_n$-orbit of the divisors $ y_1' + y_1'' + y_3 + \dots + y_n$,
where  $ y_1' + y_1'' $ is the inverse image of $x_1$, and $ y_j \mapsto x_j$.

Therefore the branch locus of $\nu$ consists of $\De$ with multiplicity
  $2^{n-1}$, and, since for $ j \geq 3$ there are two choices for $y_j$,
   of the hyperplanes $H_1, \dots, H_{2g+2}$ 
with multiplicity  $2^{n-2}$.

Concerning assertion i), observe that the element $e = \sum e_i \in V$
acts on $C^n$ via the hyperelliptic involution $\iota$ acting on each coordinate,
hence the intermediate quotient is the quotient of the symmetric product $C^{(n)}$
via the action of $\iota$.

For assertion ii), we notice that the quotient of $C^n$ by the subgroup of $K$
orthogonal to $e$ is the double covering of $(\PP^1)^n$ branched on the union of 
the branch divisors $B_C(i)$, whose image in $\PP^n$ is the union of the hyperplanes $H_j$.

\end{proof}

We can rephrase the previous result in the special case $n=g-1$:

\begin{proposition}\label{branch-theta}
If $C$ is a Hyperelliptic curve the Gauss map $\mu_C : \Theta_C \ra P: = \PP^{g-1}$
has a branch locus $\sB$ which set theoretically equals the 
union of $\De$, the Discriminant Hypersurface for divisors in $\PP^1$ 
of degree $g-1$ (the dual variety
of the rational normal curve $\Ga = \Ga_{g-1}$), with $2g+2$ Hyperplanes 
$H_1, \dots , H_{2g+2}$, where, if $z_i$ is a branch point of $ C \ra \PP^1$, then
$H_i $ corresponds to the divisors containing $z_i$.
Moreover $\De$ occurs with multiplicity $ 2^{g-2}$ in the branch locus  and the Hyperplanes 
$H_j$ occur with multiplicity $ 2^{g-3}$.

The Gauss map $\mu_C$ factors exactly through  two intermediate coverings:

i) $ \Theta_C \ra Z : = \Theta_C/ \pm 1$

ii) $ \Theta_C \ra \tilde{Z}$, where $ \tilde{Z} \ra P$ is the double cover 
branched on the union of the hyperplanes $H_i$.

\end{proposition}

\begin{proof}
We just need to observe that the hyperelliptic involution $\iota$
acts on the Jacobian $Jac (C)$ as multiplication by $-1$, for a suitable  choice of the origin
as a thetacharacteristic.

The conclusion is that a  hyperplane $H$ is in the branch locus if $H$  intersects $\Ga$
in a divisor which is the image of a canonical divisor of $C$ which 
contains a ramification point $p_i$, or contains a divisor of the form $x' + x''$, 
the inverse image of a point  $ x \in \PP^1$: this amounts to saying that 
$H$  intersects $\Ga$
in a divisor containing a branch point $z_i$ or containing a point $x$ with multiplicity at least $2$. 

\end{proof}

For further purposes, we must clarify the different roles played by the discriminant $\De$
and the union of Hyperplanes $H_1 \cup \dots \cup H_{2g+2}$ in the branch locus.

To quickly get an understanding of this issue, let us consider the case $g=n$: then the 
double covering $\tilde{Z}_g$ is a variety with trivial canonical divisor,
while $Z$ is birational to the Kummer variety of the Jacobian.
Hence the map  $Z \ra \tilde{Z}_g$ is unramifed in codimension $1$.
The main point is, as we are now going to explain, that $\De$ contributes to an exceptional divisor on the symmetric product of the curve.

\section{Theta functions on Hyperelliptic Jacobians}

Let $C$ be a curve of genus $g$, and let
$$ A : = Jac(C) = Pic^0(C).$$

Indeed, every divisor of degree $g$ is effective, and, if we fix a point $y_0 \in C$,
we have the Abel Jacobi maps 
$$ \al : C \ra Jac (C) , \ \al(y) : = \int_{y_0}^y , \ {\rm and}\ \al_n : C^n \ra Jac (C), \ \al_n (y_1, \dots, y_n ) := \sum_1^n  \int_{y_0}^{y_i}.$$

The Abel-Jacobi maps factor through the symmetric products $C^{(n)} = C^n / \frak S_n$,
and to simplify notation we shall use the same symbol for all of them.
We denote also as usual $$ W_n : = \al (C^{(n)}), n \leq g,$$
recalling once more  that $W_g = A = Jac(C)$.

For many assertions we are going to make, see \cite{acgh} pages  250 and around it.

By Riemann's singularity Theorem, if $u_0 = \al (D) , D \in C^{(g-1)}$, then there is a thetacharacteristic $\sK$ such that
$$ Mult_{u_0} (\Theta_C - \sK) = h^0 (\hol_C(D)).$$

Up to a translation, we may assume 

$$ (A, \Theta_C) = (A, W_{g-1}) , \ W_{g-1} = \al (y_0 + C^{(g-1)}) \subset \al (C^{(g)}) = A.$$

Hence the classical result that 
$$ \al : C^{(g)} \ra A$$
 is surjective, birational and locally invertible outside 
 $$\{ \al (D') | \ deg (D') = g, h^0 (\hol_C(D')) \geq 2\} \subset W_{g-1}:$$
in fact for such divisors $D'$ there exists $D'' \in |D'| $ with $D'' \geq y_0$.

\begin{corollary}\label{Thetaring1}
The graded ring of Hyperelliptic Jacobian Theta Functions
$$\sR (A, \Theta_C) : = \oplus_{m\geq 0} H^0 (A, \hol_A (m \Theta_C))$$
equals the graded ring
$$ \sR ( C^{(g)}, \al^{-1} (W_{g-1})).$$

\end{corollary}
Hence in this approach it is necessary to study the divisor 
$ \al^{-1} (W_{g-1})$, which contains the divisor $y_0 + C^{(g-1)}$.

\begin{remark}
(1) $\al ^{-1} (u)$, for $ u = \al (D)$, and $D$ an effective divisor of degree $g-1$,  is the linear system 
$|D|$, whose dimension is classically denoted by $r$. 

Since $D$ is a special divisor, it follows by Clifford's Theorem that
$ r \leq \frac{g-1}{2}$, equality holding if and only if $C$ is hyperelliptic 
and $D$ is a multiple of the hyperelliptic divisor $\frak H$.

(2) If $C$ is hyperelliptic, then $|\frak H| + C^{(g-2)} \subset C^{(g)}$
is a divisor whose image under $\al$ has dimension $g-2$.

Its intersection with $W_{g-1}$ has dimension equal to $g-3$ and is contained 
in the singular locus $Sing (W_{g-1})$.
\end{remark}

\begin{theorem}\label{thetapullback}
If $C$ is a hyperelliptic curve, then
$$  \al^{-1} (W_{g-1}) = (y_0 + C^{(g-1)}) \cup ( |\frak H| + C^{(g-2)})  = : \tilde{ C}^{(g-1)} \cup E \subset  C^{(g)},$$
where the divisor $E$ is exceptional for $\al$.
\end{theorem}

\begin{proof}
Assume that there is a divisor $\sD$ inside $C^g$ which is contracted under the Abel Jacobi map $\al$ to a lower dimensional variety.

This means that, for all $(y_1, \dots, y_g) \in \sD$, the canonical images 
$\phi(y_1), \dots, \phi(y_g) $ are linearly dependent.

After possibly reordering, $\sD$ maps onto $C^{g-1}$, and for each
$y_1, \dots, y_{g-1}$ there is a point $y$ such that $(y_1, \dots, y_{g-1}, y) \in \sD$.

For a general choice of $(y_1, \dots, y_{g-1})$, $\phi(y_1), \dots, \phi(y_{g-1}) $ are linearly independent,
span a Hyperplane $H$, and  
$$H \cap \phi(C) = \{ \phi(y_1), \dots, \phi(y_{g-1})\};$$
therefore there exists $j$ such that $\phi(y) = \phi (y_j)$,
hence the corresponding divisor on $C$ is in $|\frak H| + C^{(g-2)}$.

\end{proof}

\begin{defin}
We let 
$$ \hat{ C}_{(g-1)} : = \{ (y_1, \dots, y_g) \in C^g | \exists \ 1 \leq j \leq g, y_j = y_0\},$$
$$ \hat{ E} : = \{ (y_1, \dots, y_g) \in C^g | \exists \  j < h, y_j = \iota ( y_h) \}.$$
Here $\iota$ is the hyperelliptic involution; note that $ \hat{ C}_{(g-1)}$ maps onto 
$\tilde{ C}^{(g-1)}$, $ \hat{ E}$ maps onto $E$.
\end{defin}

For convenience, we choose now the base point $y_0 \in C$ to be a Weierstrass point, that is, a fixpoint for $\iota$: this means that $ 2 y_0 \in |\frak H|$, $2 y_0 = \psi^{-1}(x_0)$.

\begin{remark}\label{crucial}
(a) The divisor $ \hat{ C}_{(g-1)} $ is invariant for the Hypercube group $G$, 
actually $$ 2  \hat{ C}_{(g-1)} = \Psi^{-1} (H'_0) : = \Psi^{-1}  \{ (x_1, \dots, x_g)| \exists j, x_j = x_0\}=
(\pi \circ \Psi)^{-1} (H_0) ,
$$
where $H_0$ is the hyperplane in $\PP^g$ of divisors containing $x_0$.

(b) We observe here that the divisor $E$ maps onto the discriminant $\De \subset \PP^{g-1}$
under the map $\nu $.

(c) The big diagonal $\De' \subset (\PP^1)^g$, the inverse image of the Discriminant Hypersurface,
has the property that $\Psi^{-1} (\De') = \hat{E} \cup \De'_C, $
where $\De'_C$ is the big diagonal in $C^g$. $ \hat{E}$ and $\De'_C$ alone are not $G$-invariant.
\end{remark}

\begin{theorem}\label{theta}
The graded ring of Hyperelliptic Jacobian Theta Functions is a subring of invariants as follows:
$$\sR : = \sR (A, \Theta_C)  =  \sR ( C^{(g)},  \tilde{ C}^{(g-1)} + E)  = 
\sR ( C^g, \hat{ C}^{(g-1)} +  \hat{ E})^{\frak S_g} \subset$$

$$\subset  \sR ( C^g, \hat{ C}^{(g-1)} +  \Psi^{-1} (\De') ) = :\sA.$$

\end{theorem} 
\begin{proof}
The first equality is the same equality  stated in Corollary \ref{Thetaring1},
in view of Theorem \ref{thetapullback}.

For the second equality we need to observe that $\hat{ C}^{(g-1)} +  \hat{ E}$ is the pull back of the divisor  $\tilde{ C}^{(g-1)} + E$, and that the ramification divisor
of $C^g \ra C^{(g)}$ is the big diagonal $\De'_C$, mapping to the irreducible discriminant divisor $\De_C$ which is not contained in the
divisor $\tilde{ C}^{(g-1)} + E$.

Holomorphic sections downstairs (on $C^{(g)}$) clearly lift to invariant (holomorphic) sections upstairs (on $C^g$); conversely, we claim 
that  invariant sections
upstairs descend on the complement of a Zariski closed set of codimension $2$ in $C^{(g)}$,
and then they extend throughout by virtue of Hartogs' Theorem.

Our claim follows  because  on an open set of the ramification locus the pull back divisor 
is trivial, and invariant functions are pull-backs of functions on the quotient.

 For the last inclusion, we simply use that $\Psi^{-1} (\De') = \hat{E} +  \De'_C. $

\end{proof}

\begin{remark}
The graded  ring $\sA$ has the property that its subring $\sA^{even}$ is the graded ring 
associated to the pull-back $$ \Psi^{-1} (2 \De' + H'_0) = (\pi \circ \Psi)^{-1} ( H + 2 \De).$$

The ring $\sA =  \sR ( C^g, \hat{ C}^{(g-1)} +  \hat{E} +  \De'_C ) $ is a representation of the group $G$ of the Hypercube, hence 
$\sR$ is a subring of $\sA^{\frak S_g} $, and it can be detected 
by considering the subring of sections of degree $n$ vanishing of   order $n$  on the Diagonal $\De'_C$,
as done for instance by Canonaco in small genus \cite{canonaco}.

The best way to describe $\sA^{even}$ is   to write its direct image  on $(\PP^1)^g$,
but we do not pursue this further here.

\end{remark}

\section{\'Etale double covers of Hyperelliptic curves and  Jacobians}

Let $\varphi : C' \ra C$ be an \'etale double covering of a Hyperelliptic curve $C$
of genus $g$, so that 
$$ \varphi _* (\hol_{C' }) = \hol_C \oplus \eta.$$

Since the hyperelliptic involution $\iota$ acts trivially on $\Pic (C)[2]$,
$\iota$ lifts to $C'$ and we have an action of $(\ZZ/2)^2$ on $C'$ with quotient
$\PP^1$, that is, a bidouble cover of $\PP^1$.

Hence (see \cite{jdg84}) there is a factorization of the homogeneous polynomial $f$
of degree $2g+2$ whose equation is the equation for the branch locus of $\psi$, 
$$f (x) = f_1 (x) f_2 (x) \in \CC[x_0, x_1],$$
with factors of respective degrees $2 d_1, 2d_2$, with $d_1 + d_2 = g+1$, and 
such that 
$$ C' = \{ v_1^2 = f_1 (x), v_2^2 = f_2 (x) \}, \ C = \{ v^2 = f(x)\}.$$
Moreover,
$$ C = C' / j, j (v_1)=- v_1 , \  j (v_2)=- v_2, \  \varphi (x,v_1,v_2) = (x,v), {\rm with \ } v = v_1 v_2.$$
$$ H^0(C', K_{C'}) = v_2 H^0(\hol_{\PP^1}(d_1-2)) \oplus H^0(\hol_{\PP^1}(g-1))
 \oplus v_1 H^0(\hol_{\PP^1}(d_2-2)).$$
 This is the Eigenspace decomposition according to the (nontrivial) characters of $(\ZZ/2)^2$,
 and we identify $H^0(\hol_{\PP^1}(d))$ to its pull-back under $\varphi$.
 The formula clearly shows that $C'$ is hyperelliptic if and only if some $d_i =1$.
 
 We run now for $C'$ a similar game to the one we played for $C$:
 $$ (C')^n \ra C^n = (C')^n / (\ZZ/2)^n \ra (\PP^1)^n = (C')^n / ((\ZZ/2)^2)^n.$$
  The first quotient is \'etale, while if we divide by the Symmetric group $\frak S_n$,
  we get 
  $$ (C')^{(n)} \ra C^{(n)} \ra (\PP^1)^{(n)} = \PP^n.$$
  
  The map $ (C')^{(n)} \ra C^{(n)} $ above is no longer \'etale, since its degree equals $2^n$,
  but the fibre cardinality  drops over the discriminant hypersurface (for instance, if $y', y'' \mapsto y$,
  then only  three divisors $2y', 2y'' , y' + y'' $ map to the divisor $ 2y$).
  
  \begin{proposition}\label{double-etale}
  Consider the subgroup $\Lambda \subset (\ZZ/2)^n \subset  Aut ( (C')^n \ra C^n )$,
  defined as $$\Lambda : = e^{\perp} = \{ (\s_i) | \sum_i \s_i = 0\}.$$
  Then $\Lambda$ is normalized by $\frak S_n$, and 
  defining $$\hat{X}_n : =  (C')^n / (\Lambda \rtimes \frak S_n),$$
  $ \hat{X}_n$ dominates $C^{(n)}$ via an \'etale double covering.
  \end{proposition}
  
  \begin{proof}
  By the factorization 
  $$ (C')^n \ra  \hat{X}_n \ra C^{(n)} =  (C')^n / ((\ZZ/2)^n \rtimes \frak S_n),$$
  $\hat{X}_n \ra C^{(n)} $ is \'etale outside of the discriminant $\De_C$,
  and since $C^{(n)} $ is smooth, it suffices to show that the covering is quasi-\'etale, that is, 
  \'etale outside of codimension $2$.
  
  Given an effective  divisor $\sum_i m_i y_i$, where $ y'_i, y''_i \mapsto y_i$,
  we have as inverse images the effective divisors $\sum_i m'_i y'_i + m''_i y''_i$
  with $m_i = m'_i + m''_i $. For $m_i=1$, there are two possibilities, 
  for $m_i=2$, as already observed, we have three possibilities:
  $$ 2y'_i, 2y''_i , y'_i + y''_i.$$
  Assume that $m_1=2$, and all others $m_i=1$: then we simply observe, that,
  writing the divisors as images of the $n$-tuples 
   $$ (y'_1, y'_1, \dots) ,   (y''_1, y''_1, \dots), (y'_1, y''_1, \dots),$$
   there is an element of $\Sigma$, namely the involution $j_1 \times j_2 \times identity$
   which sends the first element to the second, and the third to $(y''_1, y'_1, \dots)$,
   which is equivalent modulo the action of $\frak S_n$. Hence, over  the set of divisors with $\sum_i (m_i-1) = 1$
   (whose complement has codimension $2$), the inverse image consists of two distinct points of $\hat{X}_n$.
   
  \end{proof}
  
  The previous construction is especially useful in two cases: $n = g$, where it provides an \'etale double covering
  of $Jac (C)$, which is birational to $C^{(g)} $, and for $n=g-1$, where it provides the corresponding
\'etale double covering of the Theta divisor $\Theta_C$, which is birational to $C^{(g-1)} $.

 In order to simplify the exposition, we recall the following Lemma, whose proof  can be found in \cite{andreotti}
(Proposition 3, page 806).

\begin{lemma}\label{exterior-product}
There is a natural isomorphism between the canonical system on the symmetric product of a curve
and the exterior product of the canonical system of the curve $C$
$$ \Lambda^n (H^0 (\Omega^1_C)) \cong H^0 (\Omega^n_{C^{(n)}}) = H^0 (\Omega^n_{C^{n}})^{\frak S_n},$$
associating to $\eta_1 \wedge \dots \wedge \eta_n$ the symmetrization of
$\eta_1(x_1)  \wedge \dots \wedge \eta_n(x_n)$.
\end{lemma}

  \begin{proposition}\label{double-etale-map}
  Let $\hat{X} : = \hat{X}_{g-1} $ be as in Proposition \ref{double-etale} the  \'etale double covering
  of $C^{(g-1)}$: then 
the canonical image  of   $\hat{X} \subset \PP^g$ is a finite covering    $\hat{W}$  of $\PP^{g-1}$
via a linear projection $\PP^{g} \dasharrow \PP^{g-1}$.

In terms of the two integers $d_1, d_2 \geq 1$ such that $ d_1+ d_2 = g+1$,
if  $d_1 = 1, d_2 =g$, then the  canonical image $\hat{W}$ of  $\hat{X} $ is birational to the double covering of $\PP^{g-1}$
 branched on the union of two hyperplanes $\sH : = H_1 + H_2$. 
 
 When $d_1, d_2 \geq 2$,  $\hat{W}$ is not a double covering of $\PP^{g-1}$.
  
  \end{proposition}

  \begin{proof}
  The canonical system of $\hat{X}_n $ pulls back to the $\Lambda \rtimes \frak S_n$-invariant part 
  of the canonical system of $(C')^n$, which is
  $$ H^0 (\hol_{(C')^n} (K) ) = \otimes_1^n H^0 (\hol_{C'} (K_{C'})) .$$

  By Andreotti's Lemma \ref{exterior-product} the $\frak S_n$-invariance determines a subspace 
  isomorphic to $ \Lambda^n (H^0 (\Omega^1_{C'}))$.

  We use now  the formula 
  $$ H^0(C', K_{C'}) = v_2 H^0(\hol_{\PP^1}(d_1-2)) \oplus H^0(\hol_{\PP^1}(g-1))
 \oplus v_1 H^0(\hol_{\PP^1}(d_2-2)),$$
 and replace $v_1$ by $u$, $v_2$ by $w$, so that $ uw = v$, and denoting $u(i)$ for the section $u$ on the i-th copy of $C'$, and similarly for the other variables,
 $$ H^0 (\hol_{(C')^n} (K) ) = \otimes_1^n \{  w(i) Q_i (x(i)) + P_i (x(i)) + u(i) M_i (x(i))\}.$$
 
 Set now $n= g-1$, and observe that, taking  $Q_i= M_i=0$,  that is, taking the invariants for
 $(\ZZ/2)^{g-1}$, we get the canonical system of
 $C^n$. Taking the further subring of $\frak S_{g-1}$-invariants, we get 
$$ \Lambda^{g-1} (H^0 (\Omega^1_C)) \cong H^0 (\Omega^{g-1}_{C^{(g-1)}})$$
 
  and this linear system, by Theorem \ref{andreotti} corresponds to the morphism   $C^{(g-1)}  \ra \PP^{g-1}$. 
 
 The other sections $s$ for which we are looking for must be eigenvectors  for the group 
 $\ZZ/2 = (\ZZ/2)^{g-1} / \Lambda$, and with nontrivial eigenvalue,
 hence they   must be left  invariant by $\frak S_{g-1}$ and each $\s_i$ should send them to $-s$.
 
 The second property implies that for them $P_i \equiv 0$, for all $i$.

  Hence  we get exactly one new element $v_*$, corresponding to the symmetrization of
  $$ w_1 \dots w_{d_1-1} \Lambda^{ d_1-1}H^0(\hol_{\PP^1}(d_1-2)) u_{d_1} \dots u_{g-1} \Lambda^{ d_2-1}H^0(\hol_{\PP^1}(d_2 -2)),$$ where we let $w_i : = w(i) \dots$.

  Observe  now that, if $d_1, d_2 \geq 2$, then  $Q_i, M_i \not \equiv 0$, and it is complicated to calculate
  $v_*^2$.
  
 We can however say that  $v_*$ is not an eigenvector for $((\ZZ/2)^2)^{g-1}$,
  and  $v_*^2$ as well, hence $v_*^2$  is not a section of a line bundle on $\PP^{g-1}$.

   If instead $d_1=1$, then $Q_i \equiv 0$,  then $d_2=g$, and 
   $v_*$ equals the symmetrization of 
  $$  u_{1} \dots u_{g-1} \Lambda^{ g-1}H^0(\hol_{\PP^1}(g -2)),$$
  and is a multiple of  $u_* : =   u_{1} \dots u_{g-1}$.
  
  Hence the canonical map of $\hat{X}$ factors through the double covering given by 
  $$u_*^2= f_1(x(1)) \cdot \dots \cdot f_1 (x(g-1)).$$
  
  Then we see that, setting $z_1,  z_2$ to be the roots of $f_1$, and 
$z_{3}, \dots z_{2g+2 }$ to be the roots of $f_2$,  then
  $$u_*^2= h_1 \cdot h_2 ,$$ where $ h_i$ is the linear form on $\PP^{g-1}$ whose zero set is the hyperplane $H_i$ corresponding to the symmetrization of the divisor $\{z_i\} \times (\PP^1)^{g-2}$.

  \end{proof}

\section {Application to canonical maps of hypersurfaces in Abelian Varieties}

Let $A$ be an Abelian variety of dimension $g$, and let $X \subset A$ be a smooth ample hypersurface in $A$ such that the Chern class $c_1(X)$ of
the divisor $X$ is a polarization of type $\overline{d} : = (d_1, d_2, \dots, d_{g})$, so that the vector space $H^0(A, \hol_A(X))$
has dimension equal to the Pfaffian $ d : = d_1 \cdot \dots \cdot d_g$ of $c_1(X)$. 

The classical results of Lefschetz \cite{lefschetz} say that the rational map associated to $H^0(A, \hol_A(X))$ is a morphism if $d_1\geq 2$, and is an embedding of $A$ if $d_1 \geq 3$. 

By adjunction, the canonical sheaf  of $X$ is the restriction $\hol_X(X)$, so a natural generalization of Lefschetz' theorems
is to ask about the behaviour of the canonical systems of such hypersurfaces $X$. This behaviour  depends on the hypersurface $X$ and not just  on  the polarization type only, 
as shown in   \cite{c-s}:
if we have a polarization of type $(1,1,2)$ then the image $\Sigma$ of the canonical map $\Phi_X$ is in general a surface of degree $12$ in $\PP^3$,
birational to $X$, while for the special case where $X$ is the pull-back of the Theta divisor of 
a curve of genus $3$, then the canonical map has degree $2$, and $\Sigma$ has degree $6$.

The canonical map of such a hypersurface $X$ is, via the following folklore Lemma, a mixture of the restriction of the Lefschetz map with the Gauss map of $X$, which is a morphism for $X$ smooth by a theorem of Ziv Ran \cite{ran}.

\begin{lemma}\label{canmap}
Let $X$ be an ample hypersurface of dimension $n$ in an Abelian variety $A$, such that the class of $X$ is
a polarization of type  $\overline{d} : = (d_1, d_2, \dots, d_{n+1})$.

Let $\theta_1, \dots, \theta_d$ be a basis of $H^0(A, \hol_A(X))$ such that $X = \{ \theta_1=0\}$. 

Then, if $z_1, \dots , z_g$ are linear  coordinates on the complex vector space $V$ such that $A$ is the quotient of
$V$ by a lattice $\Lam$, $A = V/ \Lam$, then the  canonical map $\Phi_X$ is given by 
$$(\theta_2, \dots, \theta_d,\frac{ \partial \theta_1}{ \partial z_1},\dots,  \frac{ \partial \theta_1}{ \partial z_g}).$$ 
\end{lemma}

Hence first of all the canonical map is an embedding if $H^0(A, \hol_A(X))$ yields an embedding of $A$; secondly, 
since a projection of $\Phi_X$ is the Gauss map of $X$, given by 
$(\frac{ \partial \theta_1}{ \partial z_1},\dots,  \frac{ \partial \theta_1}{ \partial z_g}),$
follows that 
 the canonical system $|K_X|$ is base-point-free and 
$\Phi_X$  a  finite morphism. 
 
 This is the main Theorem of \cite{Cat-Ces}:

\begin{thm}\label{genbirat}
Let $(A,X)$ be a general pair, consisting of a hypersurface $X$ of dimension $n = g-1$ in an Abelian variety $A$, such that 
the class of $X$ is
a polarization of type  $\overline{d} : = (d_1, d_2, \dots, d_{g})$ with Pfaffian $ d = d_1\dots d_g > 1$.

Then the canonical map $\Phi_X $ of $X$ is birational onto its image $\Sigma$.
\end{thm}

%%%%%%%%%%%%%%%%%%%%%%%%%%%%%%%

%%%%%%%%%%%%%%%%%%%%%%%%%%%%%
\subsection{A new proof of  theorem \ref{genbirat} }
For the reader's convenience we borrow now a simple argument   contained in \cite{Cat-Ces},
yielding first a reduction step:

{\bf Step I : It suffices to prove the  Theorem in  the case of a polarization of type  $(1,\dots, 1, p)$,
with $p$ a prime number, and assuming $g \geq 2$.}

We deal then  with the following specializations:

{\bf Step II:} {\bf Consider the  cases where $X$ is an \'etale  pull-back of a Theta divisor $\Theta$.}

Here, we shall  assume that $X$ is a polarization of type  $(1,\dots, 1, p)$,
and that $X$ is the pull-back of a Theta divisor $\Theta \subset A'$ (that is, $\Theta$  yields a principal polarization)
 via an isogeny 
$\be : A \ra A'$ with kernel $\cong \ZZ/p$.

We define 
\begin{equation}\label{Z}
Z : = \Theta / \pm 1, 
\end{equation}
and observe that $Z$
is a dihedral quotient of $X$, $ Z = X/ D_p$.

 In this situation,   4.7 of \cite{Cat-Ces} uses that the canonical system 
is a  representation of the group $D_p$ to   show  that  the canonical map
of $X$ separates the general fibres of $ X \ra Z$ for $p>2$, and that  we may assume this
also   for $p=2$ after a  deformation of $X$ (this argument shall be recalled in the final step).

{\bf Step III:} First we shall assume that $\Theta = \Theta_C $ is the Theta divisor
of a hyperelliptic curve, hence we have  $\be : A \ra Jac(C) = : A'$.

Here, there is a dihedral covering of $\PP^1$ with group $D_p$ yielding 
an unramified $\ZZ/p$ covering $C' \ra C$,
and $D_p^{g-1} \rtimes \frak S_{g-1}$ acts on $(C')^{g-1}$. 

 Observe that 
$\ZZ/p$ acts also on the canonical image, non trivially as we saw.

Assume that the canonical map
$\Phi_X : X \dasharrow  \Sigma$ is not birational, and that it factors through 
a normal variety $W$ which is birational to $\Sigma$; set then 
\begin{equation}\label{int-down}
\tilde{W} : = W / (\ZZ/p)
\end{equation}.

We have a factorization of 
$$ f \circ \Phi_X : X \ra \Theta \ra \tilde{W} \ra  \PP^{g-1},$$
where $f : \Sigma \dasharrow  \PP^{g-1}$ is the projection corresponding to the Gauss map
of $X$, which equals the Gauss map of $\Theta_C$.

By Lemma \ref{subgroups} there are four  cases possible:

\begin{enumerate}
\item
$  \tilde{W} = \Theta  $;
\item
$\Theta \ra  \tilde{W}$ has degree $2$, and $  \tilde{W} = Z  = \Theta / \pm 1$; 
\item
$ \tilde{W} = P : = \PP^{g-1}$.
\item
$ \tilde{W} \ra P$ has degree  $2$ and $ \tilde{W}  = \tilde{Z}$.
\end{enumerate}

Cases (1) and (2) are eliminated by virtue of  Step II, as follows.  

In case (1) we would have either $W=X$,
hence birationality holds, or $W = \Theta$, contradicting  Step II.

In case (2) the general  fibres  of $ X \ra W$ would be contained
in the fibres of $ X \ra Z = X/D_p$, again contradicting  Step II for $p\neq 2$.

For $p=2$, either $W = Z$, and we are done by Step II, or we have a double covering,
and by Step II a general deformation becomes birational.

In cases (3) and (4) $\tilde{W} \ra P : = \PP^{g-1}$ is either the identity or a double covering.
But, in any case, since $\ZZ/ p$ acts faithfully on the fibres of 
$\Sigma \dasharrow  \tilde{W}$,
it follows that the degree $m$ of the covering $\Sigma \dasharrow P$ (hence of $W \ra P$) is either  $2p$ or $p$.

 We have two factorizations of 
the Gauss map $f$:
$$ X \ra Z \ra P , \ X \ra W  \ra P.$$

Consider now the respective ramification divisors $\sR_f,  \sR = \sR_{Z},  \sR_W$
of the respective maps $f : X \ra P$, $\Psi : Z \ra P$, $ W \ra P$.

Since $ X \ra Z$ is quasi-\'etale (unramified in codimension $1$), $\sR_f$ is the inverse image of $\sR$,
hence  $\sR_f$  maps to $\sR$ with mapping degree $2p$.

We use now the notation and the results of Proposition \ref{branch}.
It turns out that all the components of $\sR_f$ have multiplicity $1$,
since the same happens for the components of $\sR$.

This excludes right away the case $ p \geq 3$, since a cyclic covering of degree $p$ has
a ramification divisor occurring with multiplicity $(p-1)$, and moreover $P$ and $\tilde{Z}$
do not admit unramified coverings.

Moreover, by  Step II,  we may assume that $ Z \neq \tilde{Z}$, hence that $ g \geq 4$.

We are then left with the case where $p=2$.

Here we can use Proposition \ref{double-etale}, first under the assumption that
we choose $d_1=1 , d_2 =g$. 

 The image of $X$, which equals the one of $\hat{X}$,
is the double cover $\hat{W}$ of $P$ with branch locus $H_1 + H_2$,  the union 
of  $2$ Hyperplanes.

Since $W = \hat{W}$  case (4) is clearly excluded, since  $\hat{W}$ is  not a double  covering of $\tilde{Z}$.

Use now Proposition \ref{double-etale} under the assumption that
 $d_1 , d_2 \geq 2$.
Then case  (3), where    $W $ would be the double covering
of $P= \PP^{g-1}$ branched on a branch divisor $\sB' \subset  \De \cup \sH$, 
is also excluded.

\bigskip

{\bf Step IV:} 
To finish the proof, consider the more general  case where $X$ is a double \'etale covering of a smooth Theta divisor $\Theta$.

As observed in \cite{c-s}, we have a basis $\theta_1, \theta_2$ of even functions,
i.e., such that $\theta_i (-z) = \theta_i(z)$, and  $Z : = \Theta/ \pm 1 = X / (\ZZ/2)^2$, where $(\ZZ/2)^2$
acts sending $ z \mapsto \pm z + \eta$, where $\eta$ is a $2$-torsion point on $A$.
Then the canonical map  $\Phi_X$, since the partial derivatives of $\theta_1$ are invariant for $ z \mapsto  z + \eta$,
while $\theta_2 (z + \eta) = - \theta_2 (z)$, factors through the involution $\iota : X \ra X$
such that
$$ \iota (z) =  -z + \eta.$$
Assume that we have a further factorization $ X \ra X/\iota \ra \Sigma  $ of the canonical map,
and recall  that $X/\iota$ is a double covering of $Z$.

Then, specializing to the case where we have the  double \'etale covering $X_0$ of the  Theta divisor  of a hyperelliptic curve, we see 
by Theorem \ref{blocks} that
we have a further factorization of the canonical map of $X_0$, $X_0 \ra \Sigma_0 \ra \hat{W}$.

 Indeed, Hypothesis (**) can be seen to hold using Theorem \ref{thetapullback}
and assertion (v) of Proposition \ref{smooth}.
 
As we argued before, $\Sigma_0 $  is a double cover of $\Sigma_0 / (\ZZ/2)$, which is
therefore either $Z$ or $\tilde{Z}$.

Accordingly, either  $\Sigma = X/\iota$ or $\Sigma_0 / (\ZZ/2) =  \tilde{Z}$, hence  $\Sigma_0$ is a degree
four covering of $P$.

In the  latter case, if we take the degrees $d_1, d_2 \geq 2$, it would follow that $\Sigma_0 =  \hat{W}$.
 
Hence  the monodromy of $\Sigma \ra P$ would land in $\frak S_4$.

However, if we specialize to  the Theta divisor of a non hyperelliptic curve, the Monodromy  group
of the covering $ C^{(g-1)} \sim \Theta_C \ra P$  is equal to $\frak S_{2g-2}$,
acting on the subsets of cardinality $(g-1)$. 

Since $ g \geq 4$, $2g-2 \geq 6$ and the group $\frak A_{2g-2}$
is simple: hence the monodromy image in $\frak S_4$ has order $2$, and cannot be transitive, whence a contradiction.

\subsection{Final step}
We repeat here the  final argument which  takes care of the case $\Sigma = X/\iota$, as in \cite{c-s}: if for a general deformation of $X$ as a symmetric divisor the canonical map would factor through $\iota$,
then $X$ would be $\iota$-invariant; being symmetric, it would be $(\ZZ/2)^2$-invariant, hence for all deformations
$X$ would remain the pull-back of a Theta divisor. This is a contradiction, since the Kuranishi family 
of $X$ has  higher dimension than the Kuranishi family  of a Theta divisor $\Theta$ (see \cite{c-s}).

\bigskip

 {\bf Acknowledgements:} the  author would like to  thank Luca Cesarano, Edoardo Sernesi and especially Ciro Ciliberto    for some  interesting conversation.
 
  He is extremely thankful to the referee for a careful reading of the manuscript, and for pointing out two obscure arguments,
 which were indeed incorrect, and needed a reparation.

%%%%%%%%%%%%%%%%%%%%%%%%%%%%%%%%%%%%%%%%%%%%%%%%%%%%%%%%%%%%%%%%%%%%%%%%%%%%%%%%%%%%%%%%%%%%%%%%%%%%%%%%%%%%%%%%%%%%%%%%%%%%
%%%%%%%%%%%%%%%%%%%%%%%%%%%%%%%%%%%%%%%%%%
%%%%%%%%%%%%%%%%%%%%%%%%%%%%%%%%%%%%%%%%%%

\end{document}